\documentclass[1pt]{article}
\usepackage{amsmath, amssymb, amsthm, amsfonts, cases}
\usepackage{mathrsfs}
\usepackage{url}
\usepackage{authblk}
\usepackage[usenames]{color}
\usepackage{geometry}
\theoremstyle{plain}
\newtheorem{thm}{Theorem}[section]
\newtheorem{cor}[thm]{Corollary}
\newtheorem{pro}[thm]{Problem}
\newtheorem{lem}[thm]{Lemma}

\theoremstyle{definition}
\newtheorem{defn}{Definition}[section]
\newtheorem{ass}{Assumption}[section]
\newtheorem{rmk}{Remark}[section]

\makeatletter\@addtoreset{equation}{section} \makeatother

\begin{document}

\title{   Maximum Principle for Partial Observed  Zero-Sum Stochastic
Differential Game  of Mean-Field SDEs
\thanks{This work was supported by the Natural Science Foundation of Zhejiang Province
for Distinguished Young Scholar  (No.LR15A010001),  and the National Natural
Science Foundation of China (No.11471079, 11301177) }}

\date{}

\author[a]{Maoning Tang}
\author[a]{Qingxin Meng\footnote{Corresponding author.
\authorcr
\indent E-mail address: mqx@zjhu.edu.cn(Q. Meng),}}
\affil[a]{\small{Department of Mathematical Sciences, Huzhou University, Zhejiang 313000, China}}
\maketitle

\begin{abstract}

In this paper, we consider  a partial  observed   two-person zero-sum stochastic
differential game problem where the system is governed by  a
stochastic differential  equation of mean-field type.  Under  standard assumptions on the coefficients,  the maximum principles for optimal open-loop control in a strong sense as well as a weak one are established  by the associated optimal control  theory in Tang and Meng (2016).
 To illustrate the general results, a class of  linear quadratic stochastic differential game problem  is discussed  and the existence and
dual characterization  for the partially observed  open-loop saddle are  obtained.
\end{abstract}

\section{Introduction}
Recently, thanks to many practical applications such as in economics and finance, stochastic optimal control problems for stochastic differential equation (SDE) of   mean-field type have been extensively studied and  a  number of  important theoretical and practical application results are
      obtained under full observations or partial observations.   As Djehiche and Tembine (2016) stated in their paper, the speciality of the
optimal control problem of mean-field type
      is that  the coefficients of the stochastic  state equation and   cost functional are
 dependent not only on the state and the control,  but also on their probability distribution. The presence of  the mean-field term makes the control problem   time-inconsistent and  the dynamic programming principle (DPP) ineffective, which motivates to establish
   the stochastic maximum  principle (SMP) to
    solve this type of optimal control problems instead of trying extensions of DPP.
    And the so-called full observed optimal control problem is that  in this case the controller has full complete information filtration available on the admissible control.
       We refer to interested readers to Andersson and Djehiche (2011), Buckdahn et al (2011),
Li (2012),  Meyer-Brandis et al(2012),
Shen
and Siu (2013), Du et al (2013), Elliott(2013),
Hafayed (2013), Yong (2013),
Chala(2014), Shen et al (2014), Meng and Shen(2015)
 and  the reference therein for the various optimal control theory results on the mean-field models under full observation. On the other hand,  for the partial observed or
 partial information optimal control problem, the objective is to find an optimal
control for which the controller has less information than the complete information filtration. In particular, sometimes an economic model in which there are information gaps
among economic agents can be formulated as a partial information optimal control problem
(see {\O}ksendal, B. (2006), Kohlmann and Xiong (2007) ). A great of results on stochastic
optimal control
without mean-field  term  under   partial
observation  or partial  information
have  been obtained by  many authors
for various types of stochastic systems
via establishing the corresponding   MP and DPP. See e.g., Bensoussan(1983), Tang (1998), Baghery et al. (2007), Wu(2010), Wang and  Wu(2009),
Wang et al(2013, 2015a), and the reference therein for
more detailed discussion.
  Recently,  for partial observed stochastic optimal control problem of stochastic systems
   of mean-field type,  due to the theoretical and practical interest, it become more popular, e.g,
Wang et al (2014a, 2014b, 2015b, 2016),
 Djehiche and
 Tembine (2016), Ma and Liu (2017),
where the corresponding
maximum principles
are established and  practical
finance applications are illustrated.

  This purpose of this paper is to
  extend  the optimal control problem to  two-person zero-sum differential game problem for SDE of mean-field type under partial observation.
  As stated in Basar (2010), game theory deals with strategic interactions among multiple decision makers, called players (and in
some context agents), with each player's preference ordering among multiple alternatives captured
in an objective function for that player, which she either tries to maximize (in which case the
objective function is a utility function or benefit function) or minimize (in which case we refer to
the objective function as a cost function or a loss function).
Specially, we say that  a game problem has the zero-sum property which means that there is a single performance criterion which one player tries to minimize and the other tries to maximize.
Differential game theory investigates conflict problems in systems which are driven by differential equations. This topic lies at the intersection of game theory (several players are involved) and of controlled systems (the differential equations are controlled by the players). The theory of differential games was initiated by Issacs(1954). It was later studied in greater detail by Fleming and Berkovitz (1955). After
the development of Pontryagin's maximum principle, it became clear that there was a connection between differential games and optimal control
theory. In fact, differential game problems represent a generalization of optimal control problems in cases where there are more than one controller
or player. However, differential games are conceptually far more
complex than optimal control problems in the sense that it is no longer
obvious what constitutes a solution; indeed, there are a number of different types of solutions
such as minimax solutions for zerosum
differential games,  Nash solutions for nonzero-sum
game problems , Stackelberg differential games, along with possibilities of cooperation and bargaining,  see  Basar and Olsder (1999), Sethi and Thompson(2000). For the partial information stochastic differential
games,   recently, An and
${\O}$ksendal (2008)  established a maximum principle
for forward systems with Poisson jumps.  Moreover,  we refer to Wang and Yu (2010, 2012), Meng and  Tang(2010), Xiong et al (2016) and the references therein for more associated
results  on the partial information or partial observed    differential games for
all kinds of different stochastic systems without
mean-field terms.

But to our best knowledge,  there is
few discussions on the
 partial observed stochastic differential games problem for the  stochastic system of mean-field type, which motives us to write
this paper.
The main contribution of this paper  is to establish the partial observed maximum principle in the weak formulation and strong formulation for
the optimal open-loop control under our two-person zero-sum differential games framework  by  using  the results in Tang and Meng (2016). The
results obtained in this paper can be considered as a generalization of stochastic optimal control problem of mean-field type  to the two-person zero-sum
case under partial observations.  As an application, a  two-person zero-sum
stochastic differential game
of linear mean-field stochastic differential equations with a quadratic
cost criteria under partial observation is discussed where the existence of the corresponding
open-loop saddle is obtained and the
optimal control is characterized explicitly by adjoint processes.

The rest of this paper is organized as follows. We
introduce useful notations
and formulate the two-person zero-sum
differential game problem for
mean-field SDEs under partial
observation.  Section 2 is devoted to  the necessary maximum principle in a weak formulation  for optimal open-loop control
 by the optimal control theory in Tang and Meng (2016).   In Section 3, necessary as well as sufficient optimality
conditions for our differential games
in a strong formulation is derived.
 As an application, a class of  linear quadratic stochastic differential game problem under partial observations  is studied.

\section{Basic Notations}
In this section, we introduce
some basic notations  which will be
used in this paper.
Let ${\cal T} : = [0, T]$ denote a finite time index, where $0<T <
\infty$. We consider a complete probability space $( \Omega,
{\mathscr F}, {\mathbb P} )$  equipped with two one-dimensional
standard Brownian motions $\{W(t), t \in {\cal T}\}$ and $\{Y(t),t \in {\cal T}\},$
respectively. Let $%
\{\mathscr{F}^W_t\}_{t\in {\cal T}}$ and $%
\{\mathscr{F}^Y_t\}_{t\in {\cal T}}$ be $\mathbb P$-completed natural
filtration generated by $\{W(t), t\in {\cal T}\}$ and $\{Y(t), t\in {\cal T}\},$ respectively. Set $\{\mathscr{F}_t\}_{t\in {\cal T}}:=\{\mathscr{F}^W_t\}_{t\in {\cal T}}\bigvee
\{\mathscr{F}^Y_t\}_{t\in {\cal T}}, \mathscr F=\mathscr F_T.$   Denote by $\mathbb E[\cdot]$ the expectation
under the probablity $\mathbb P.$
 Let $E$ be a Euclidean space. The inner product in $E$ is denoted by
$\langle\cdot, \cdot\rangle,$ and the norm in $ E$ is denoted by $|\cdot|.$
Let $A^{\top }$ denote  the
transpose of the matrix or vector $A.$
For a
function $\phi:\mathbb R^n\longrightarrow \mathbb R,$ denote by
$\phi_x$ its gradient. If $\phi: \mathbb R^n\longrightarrow \mathbb R^k$ (with
$k\geq 2),$ then $\phi_x=(\frac{\partial \phi_i}{\partial x_j})$ is
the corresponding $k\times n$-Jacobian matrix. By $\mathscr{P}$ we
denote the
predictable $\sigma$ field on $\Omega\times [0, T]$ and by $\mathscr %
B(\Lambda)$ the Borel $\sigma$-algebra of any topological space
$\Lambda.$ In the follows, $K$ represents a generic constant, which
can be different from line to line.

 Next we introduce some spaces of random variable and stochastic
 processes. For any $\alpha, \beta\in [1,\infty),$ we let

$\bullet$~~$M_{\mathscr{F}}^\beta(0,T;E):$ the space of all $E$-valued and ${%
\mathscr{F}}_t$-adapted processes $f=\{f(t,\omega),\ (t,\omega)\in {\cal T}
\times\Omega\}$ satisfying
$
\|f\|_{M_{\mathscr{F}}^\beta(0,T;E)}\triangleq{\left (\mathbb E\bigg[\displaystyle%
\int_0^T|f(t)|^ \beta dt\bigg]\right)^{\frac{1}{\beta}}}<\infty. $

$\bullet$~~$S_{\mathscr{F}}^\beta (0,T;E):$ the space of all $E$-valued and ${%
\mathscr{F}}_t$-adapted c\`{a}dl\`{a}g processes $f=\{f(t,\omega),\
(t,\omega)\in {\cal T}\times\Omega\}$ satisfying $
\|f\|_{S_{\mathscr{F}}^\beta(0,T;E)}\triangleq{\left (\mathbb E\bigg[\displaystyle\sup_{t\in {\cal T}}|f(t)|^\beta \bigg]\right)^{\frac
{1}{\beta}}}<+\infty. $

$\bullet$~~$L^\beta (\Omega,{\mathscr{F}},P;E):$ the space of all
$E$-valued random variables $\xi$ on $(\Omega,{\mathscr{F}},P)$
satisfying $ \|\xi\|_{L^\beta(\Omega,{\mathscr{F}},P;E)}\triangleq
\sqrt{\mathbb E|\xi|^\beta}<\infty. $

$\bullet$~~$M_{\mathscr{F}}^\beta(0,T;L^\alpha (0,T; E)):$ the space of all $L^\alpha (0,T; E)$-valued and ${%
\mathscr{F}}_t$-adapted processes $f=\{f(t,\omega),\ (t,\omega)\in[0,T]%
\times\Omega\}$ satisfying $
\|f\|_{\alpha,\beta}\triangleq{\left\{\mathbb E\bigg[\left(\displaystyle
\int_0^T|f(t)|^\alpha
dt\right)^{\frac{\beta}{\alpha}}\bigg]\right\}^{\frac{1}{\beta}}}<\infty. $

\subsection{Formulation of  Two-Person Zero-Sum Differential Games of Mean-Field Type Under  Partial Observation }

In the following,  we formulate  a partial observed two-person zero-sum stochastic
differential game problem  in a weak form and a strong form, respectively, where the system is governed by the following
nonlinear  mean-field stochastic differential equation
\begin{equation}
\displaystyle\left\{
\begin{array}{lll}
dx(t) & = & b(t,x(t),\mathbb E [ x(t)],u_1(t),
\mathbb E[u_1(t)], u_2(t), \mathbb E[u_2(t)])dt
\\&&+\displaystyle
g(t,x(t),\mathbb E [ x(t)],u_1(t),
\mathbb E[u_1(t)], u_2(t), \mathbb E[u_2(t)] )dW(t)
\\&&+\displaystyle {\tilde g}(t,x(t),\mathbb E [ x(t)], u_1(t),
\mathbb E[u_1(t)], u_2(t), \mathbb E[u_2(t)] )dW^{(u_1,u_2)}(t), \\
\displaystyle x(0) & = & a\in \mathbb{R}^n,
\end{array}%
\right.  \label{eq:1.1}
\end{equation}%
with an obvervation
\begin{equation}\label{eq:1.2}
\displaystyle\left\{
\begin{array}{lll}
dY(t) & = & h(t,x(t),\mathbb E [ x(t)], u_1(t),
\mathbb E[u_1(t)], u_2(t), \mathbb E[u_2(t)] )dt+dW^{(u_1, u_2)}(t),  \\
\displaystyle y(0) & = & 0,
\end{array}
\right.
\end{equation}
where $b: {\cal T} \times \Omega \times {\mathbb R}^n \times
{\mathbb R}^n
 \times U_1\times U_1\times U_2\times U_2\rightarrow {\mathbb R}^n$, $g:
{\cal T} \times \Omega \times {\mathbb R}^n \times {\mathbb R}^n
 \times U_1\times U_1\times U_2\times U_2\rightarrow {\mathbb R}^n$, $\tilde g: {\cal T} \times \Omega \times {\mathbb R}^n \times {\mathbb R}^n
 \times U_1 \times U_1\times U_2\times U_2\rightarrow {\mathbb R}^n $, $h: {\cal T} \times \Omega \times {\mathbb R}^n \times {\mathbb R}^n
 \times U_1 \times U_1\times U_2\times U_2 \rightarrow {\mathbb R}$,  are given random mapping with ${
U}_1\subset R^{k_1}$ and ${U}_2\subset R^{k_2}$  being two given nonempty convex sets. In the above, $u_1(\cdot)$ and $u_2(\cdot)$ are partial observed stochastic processes called admissible control processes of Player 1 and Player 2 respectively defined as follows.

 \begin{defn}
     For $i=1,2,$ a partial observed   admissible control process for Player $i$ is defined as a
   stochastic  process $u_i:{\cal T}\times  \Omega\longrightarrow  U_i$ which is
   $\{\mathscr{F}^Y_t\}_{t\in \cal T}$
-adapted   and satisfies
\begin{equation} \label{eq:1.33}
  \mathbb E\bigg[\bigg(\int_0^T|u_i(t)|^2
dt\bigg)^{2}\bigg]<\infty.
\end{equation}
The set of all admissible controls
is denoted by ${\cal  A}_{i}^W, i=1,2.$
The control pair $(u_1(\cdot), u_2(\cdot))
\in {\cal  A}_{1}^W \times {\cal  A}_{2}^W$
is called a pair of admissible controls of  players.
\end{defn}

Now we make the following standard  assumptions
 on the coefficients of the equations \eqref{eq:1.1}
 and \eqref{eq:1.2}.

 \begin{ass}\label{ass:1.1}
The coefficients $b$, $g,\tilde g$ and $h$  are ${\mathscr P} \otimes {\cal
B} ({\mathbb R}^n) \otimes {\mathscr B} ({\mathbb R}^n) \otimes {\mathscr B}
(U_1) \otimes {\mathscr B}
(U_1)\otimes {\mathscr B}
(U_2) \otimes {\mathscr B}
(U_2)  $-measurable. For each $(x,y,
u_1,v_1, u_2, v_2) \in \mathbb {R}^n\times \mathbb R^n \times U_1 \times U_1\times U_2\times U_2$, $b (\cdot, x,y,
u_1,v_1, u_2, v_2), g(\cdot,x,y,u_1, v_1, u_2,v_2)$, $\tilde g (\cdot, x, y, u_1,v_1, u_2,v_2)$
and $h(\cdot, x, y, u_1,v_1,u_2,v_2)$ are all $\{\mathscr{F}_t\}_{t\in \cal T}$-adapted processes. For almost all $(t, \omega)\in {\cal T}
\times \Omega$, the mapping
\begin{eqnarray*}
(x,y,u_1,v_1,u_2,v_2) \rightarrow \varphi(t,\omega,x, y, u_1,v_1,u_2,v_2)
\end{eqnarray*}
 is continuous differentiable with respect to $(x, y,u_1,v_1,u_2,v_2)$ with
appropriate growths, where $\varphi=b, g,  \tilde g$ and $h.$  More precisely, there exists a constant $C
> 0$  such that for  all $x,y\in \mathbb
R^n, u_1,v_1\in U_1,u_2,v_2\in U_2$ and a.e. $(t, \omega)\in {\cal T} \times \Omega,$
\begin{eqnarray*}
\left\{
\begin{aligned}
& (1+|x|+|y|+|u_1|+|v_1|+|u_2|+|v_2|)^{-1}
|\phi(t,x,y,u_1,v_1,u_2,v_2)|
+|\phi_x(t,x,y,u_1,v_1,u_2,v_2)|
\\&\quad\quad+|\phi_y(t,x,y,u_1,v_1,u_2,v_2)|
+|\phi_{u_1}(t,x,y,u_1,v_1,u_2,v_2)|
+|\phi_{v_1}(t,x,y,u_1,v_1,u_2,v_2)|
\\&\quad\quad+|\phi_{u_2}(t,x,y,u_1,v_1,u_2,v_2)|
+|\phi_{v_2}(t,x,y,u_1,v_1,u_2,v_2)|\leq C, \varphi=b, g,\tilde g;
\\&|h(t,x,y,u_1,v_1,u_2,v_2)| +|h_{x}(t,x,y,u_1,v_1,u_2,v_2)|
+|h_{y}(t,x,y,u_1,v_1,u_2,v_2)|
+|h_{u_1}(t,x,y,u_1,v_1,u_2,v_2)|
\\&\quad\quad+|h_{v_1}(t,x,y,u_1,v_1,u_2,v_2)|
+|h_{u_2}(t,x,y,u_1,v_1,u_2,v_2)|
+|h_{v_2}(t,x,y,u_1,v_1,u_2,v_2)|\leq
C.
\end{aligned}
\right.
\end{eqnarray*}
\end{ass}
Now under Assumption \ref{ass:1.1}, we
begin to discuss  the well- posedness
of \eqref{eq:1.1} and \eqref{eq:1.2}.
Indeed, putting \eqref{eq:1.2} into the state equation \eqref{eq:1.1}, we get that
\begin{equation} \label{eq:1.3}
\displaystyle\left\{
\begin{array}{lll}
dx(t) & = & [(b-{\tilde g}h)(t,x(t),\mathbb E [ x(t)],u_1(t),
\mathbb E[u_1(t)], u_2(t), \mathbb E[u_2(t)])]dt
\\&&+\displaystyle
g(t,x(t),\mathbb E [ x(t)],u_1(t),
\mathbb E[u_1(t)], u_2(t), \mathbb E[u_2(t)])dW(t)\\&&
 +\displaystyle {\tilde g}(t,x(t),\mathbb E [ x(t)],u_1(t),
\mathbb E[u_1(t)], u_2(t), \mathbb E[u_2(t)])dY(t),
\\
\displaystyle x(0) & = & a.
\end{array}
\right.
\end{equation}
Under Assumption \ref{ass:1.1}, for any $(u_1(\cdot), u_2(\cdot))
\in {\cal  A}_{1}^W \times {\cal  A}_{2}^W,$  by
Lemma 1.4 in Tang and Meng (2016),
\eqref{eq:1.3} admits a strong solution
$x(\cdot)\equiv x^{(u_1,u_2)}(\cdot)\in S^4_{\mathscr F}(0, T; \mathbb R^n).$  On the other hand, for any $(u_1(\cdot), u_2(\cdot))
\in {\cal  A}_{1}^W \times {\cal  A}_{2}^W,$ associated with
the corresponding solution $x^{(u_1,u_2)}(\cdot)$ of
\eqref{eq:1.3},  introduce  a stochastic
process $Z^{(u_1,u_2)}(\cdot)$  defined by the
unique solution  of the following mean-field
SDE \begin{equation}\label{eq:3.4}
\displaystyle\left\{
\begin{array}{lll}
dZ^{(u_1,u_2)}(t) & = & Z^{(u_1,u_2)}(t)h(t,x^{(u_1,u_2)}(t), \mathbb E [ x^{(u_1,u_2)}(t)],u_1(t),
\mathbb E[u_1(t)], u_2(t), \mathbb E[u_2(t)] )dY(t), \\
\displaystyle Z^{(u_1,u_2)}(0) & = & 1.
\end{array}%
\right.
\end{equation}
Define a new probability measure $\mathbb P^{(u_1,u_2)}$ on $(\Omega, \mathscr F)$ by
$d\mathbb P^{(u_1,u_2)}=Z^{(u_1,u_2)}(1)d\mathbb P.$  Then from Girsanov's theorem and \eqref{eq:1.2}, $(W(\cdot),W^{(u_1,u_2)}(\cdot))$ is an
$\mathbb R^2$-valued standard Brownian motion defined in the new probability
space $(\Omega, \mathscr F, \{\mathscr{F}_t\}_{0\leq t\leq T},\mathbb P^{(u_1,u_2)}).$
So $(\mathbb P^{(u_1,u_2)}, X^{(u_1,u_2)}(\cdot), Y(\cdot), W(\cdot), W^{(u_1,u_2)}(\cdot))$ is a weak
solution on $(\Omega, \mathscr F, \{\mathscr{F}_t\}_{t\in \cal
T})$ of  \eqref{eq:1.1} and
\eqref{eq:1.2}.

Now for  any given pair of  admissible control
 $(u_1(\cdot), u_2(\cdot))
\in {\cal  A}_{1}^W \times {\cal  A}_{2}^W,$   and the  corresponding
 weak solution  $(\mathbb P^{(u_1,u_2)},x^{(u_1,u_2)}(\cdot),  Y(\cdot), W(\cdot), W^{(u_1,u_2)}(\cdot))$ of \eqref{eq:1.1} and
 \eqref{eq:1.2}, we introduce the following  cost functional in the weak form,
\begin{equation}\label{eq:1.6}
\begin{split}
J(u_1(\cdot ), u_2(\cdot))=& \mathbb E^{(u_1,u_2)}\displaystyle\bigg[
\int_{0}^{T}l(t,x(t),\mathbb E [ x(t)], u_1(t),
\mathbb E[u_1(t)], u_2(t), \mathbb E[u_2(t)])dt \\
& ~~~~~+m(x(T),\mathbb E [ x(T)])\bigg].
\end{split}%
\end{equation}
where $\mathbb E^{(u_1,u_2)}$ denotes the expectation with respect to the
probability space $(\Omega, \mathscr F, \{\mathscr{F}_t\}_{0\leq
t\leq T},\mathbb P^{(u_1, u_2)})$ and  $l:
{\cal T} \times \Omega \times {\mathbb R}^n \times {\mathbb R}^n
 \times U_1 \times U_1\times U_2\times U_2\rightarrow {\mathbb R},$ $m: \Omega \times {\mathbb R}^n \times {\mathbb R}^n \rightarrow {\mathbb R}$  are given random mappings
satisfying  the following assumption:

\begin{ass}\label{ass:1.2}
 $l$ is ${\mathscr P} \otimes {\cal
B} ({\mathbb R}^n) \otimes {\mathscr B} ({\mathbb R}^n) \otimes {\mathscr B}
(U_1) \otimes {\mathscr B}
(U_1)\otimes {\mathscr B}
(U_2) \otimes {\mathscr B}
(U_2) $-measurable, and $m$ is ${\cal F}_T \otimes {\mathscr B} ({\mathbb
R}^n) \otimes {\mathscr B} ({\mathbb R}^n)$-measurable. For each $(x,y,
u_1,v_1, u_2,v_2) \in \mathbb {R}^n\times \mathbb R^n \times U_1\times U_1\times U_2\times U_2$, $f (\cdot, x, y, u_1,v_1, u_2, v_2)$ is  an ${\mathbb F}$-adapted process, and $m(x,y)$
is an ${\cal F}_{T}$-measurable random variable . For almost all $(t, \omega)\in [0,T]
\times \Omega$, the mappings
\begin{eqnarray*}
(x,y,u_1,v_1, u_2,v_2) \rightarrow l(t,\omega,x, y, u_1,v_1,u_2,v_2)
\end{eqnarray*}
and
\begin{eqnarray*}
(x,y) \rightarrow m(\omega,x, y)
\end{eqnarray*}
are continuous differentiable with respect to $(x, y,u_1,v_1, u_2, v_2)$ with
appropriate growths, respectively.
More precisely, there exists a constant $C
> 0$  such that for  all $x,y\in \mathbb
R^n, u_1,v_1\in U_1, u_2,v_2 \in U_2$ and a.e. $(t, \omega)\in [0,T]
\times \Omega,$
\begin{eqnarray*}
\left\{
\begin{aligned}
&
(1+|x|+|y|+|u_1|+|v_1|+|u_2|+|v_2|)^{-1} (|l_x(t,x,y,u_1,v_1,u_2,v_2)|
+|l_y(t,x,y,u_1,v_1,u_2,v_2)|
\\&\quad\quad+|l_{u_1}(t,x,y,u_1,v_1, u_2,v_2)|
+|l_{v_1}(t,x,y,u_1,v_1,u_2,v_2)|
+|l_{u_2}(t,x,y,u_1,v_1,u_2,v_2)|+|l_{v_2}(t,x,y,u_1,v_1,u_2,v_2)|)
\\&\quad\quad+(1+|x|^2+|y|^2+|u_1|^2+|v_1|^2
+|u_2|^2+|v|^2)^{-1}|l(t,x,y,u_1,v_1,u_2,v_2)|
 \leq C;
\\
& (1+|x|^2+|y|^2)^{-1}|m(x,y)| +(1+|x|+|y|)^{-1}(|m_x(x,y)|+|m_y(x,y)|)\leq
C.
\end{aligned}
\right.
\end{eqnarray*}
\end{ass}
 Under  Assumption \ref{ass:1.1} and
 \ref{ass:1.2},
 by the estimates in Lemma 1.4 in Tang and Meng (2016),  we get  that
the cost functional is well-defined.

Then we  can put forward the following partially observed  two-person zero-sum differential  game problem  in its weak formulation,
 i.e., with changing
 the reference probability space $(\Omega, \mathscr F, \{\mathscr{F}_t\}_{0\leq
t\leq T},\mathbb P^u),$ as follows.

\begin{pro}
\label{pro:1.1}
Find a pair of admissible control pair
$(\bar{u}_1(\cdot), \bar{u}_2(\cdot))$ over  ${\cal A}_1^W\times{\cal A}_2^W$
such that
\begin{equation}\label{c1}
J(\bar{u}_1(\cdot), \bar{u}_2(\cdot))=\sup_{u_2(\cdot)\in {\cal
A}_2^W}\left(\inf_{u_1(\cdot)\in {\cal
A}_1^W}J(u_1(\cdot),u_2(\cdot))\right)=\inf_{
u_1^W(\cdot)\in {\cal
A}_1^W}\left(\inf_{u_2(\cdot)\in {\cal
A}_2^W}J(u_1(\cdot),u_2(\cdot))\right),
\end{equation}
subject to the
 state equation \eqref{eq:1.1}, the
 observation equation \eqref{eq:1.2}
 and the cost functional \eqref{eq:1.6}.

\end{pro}

Obviously, according to  Bayes' formula,
 the cost functional \eqref{eq:1.6} can be rewritten as
\begin{equation}\label{eq:1.7}
\begin{split}
J(u_1(\cdot ), u_2(\cdot))=& \mathbb E\displaystyle\bigg[%
\int_{0}^{T}Z^{(u_1, u_2)}(t)l(t,x(t),\mathbb E [ x(t)], u_1(t),
\mathbb E[u_1(t)], u_2(t), \mathbb E[u_2(t)])dt \\
& ~~~~~+Z^{(u_1,u_2)}(T)m(x(T),\mathbb E [ x(T)])\bigg].
\end{split}
\end{equation}
  Therefore, we  can translate   Problem \ref{pro:1.1}
  into  the following  equivalent optimal control problem in
  its strong formulation, i.e., without changing the  reference
probability space $(\Omega, \mathscr F, \{\mathscr{F}_t\}_{0\leq t\leq T},\mathbb P),$
where $Z^{(u_1,u_2)}(\cdot)$ will be regarded as
an additional  state process besides the
state process $x^{(u_1,u_2)}(\cdot).$

\begin{pro}
\label{pro:1.2} Find an admissible open-loop control
$(\bar{u}_1(\cdot), \bar{u}_2(\cdot))$ over  ${\cal A}_1^W\times{\cal A}_2^W$
such that
\begin{equation}\label{eq:2.9}
J(\bar{u}_1(\cdot), \bar{u}_2(\cdot))=\sup_{u_2(\cdot)\in {\cal
A}_2^W}\left(\inf_{u_1(\cdot)\in {\cal
A}_1^W}J(u_1(\cdot),u_2(\cdot))\right)
=\inf_{
u_1(\cdot)\in {\cal
A}_1^W}\left(\inf_{u_2(\cdot)\in {\cal
A}_2^W}J(u_1(\cdot),u_2(\cdot))\right),
\end{equation}
 subject to
 the cost functional \eqref{eq:1.7}
 and  the following
 state equation
\begin{equation}
\displaystyle\left\{
\begin{array}{lll}
dx(t) & = & [(b-{\tilde
g}h)(t,x(t),
u_1(t),
\mathbb E[u_1(t)], u_2(t), \mathbb E[u_2(t)]]dt
\\&&+\displaystyle
g(t,x(t),\mathbb E [x(t)],u_1(t),
\mathbb E[u_1(t)], u_2(t), \mathbb E[u_2(t)])dW(t)
 \\&&+\displaystyle {\tilde g}(t,x(t),\mathbb E[ x(t)],u_1(t),
\mathbb E[u_1(t)], u_2(t), \mathbb E[u_2(t)])dY(t),
 \\
dZ(t) & = & Z(t)h(t,x(t),\mathbb E [ x(t)],u_1(t),
\mathbb E[u_1(t)], u_2(t), \mathbb E[u_2(t)])dY(t), \\
\displaystyle Z(0) & = & 1,
\\
\displaystyle x(0) & = & a\in \mathbb{R}^n.
\end{array}%
\right.  \label{eq:3.7}
\end{equation}
\end{pro}
Any $(\bar{u}_1(\cdot), \bar{u}_2(\cdot))
 \in {\cal A}_1^W\times {\cal A}_2^W$ satisfying
 \eqref{eq:2.9} is  called an
optimal open-loop control. The corresponding strong solution  $(\bar{x}(\cdot), \bar Z(\cdot)) $ of
 \eqref{eq:3.7} is called   the optimal state process.  Then
$(\bar{u}_1(\cdot),\bar{u}_2(\cdot); \bar{x}(\cdot), \bar{Z}(\cdot))$ is
called an optimal pair.

Roughly speaking, for the zero-sum  differential game, Player I seek
control $\bar {u}_1(\cdot)$ to minimize \eqref{eq:1.7}, and Player II seek
control $\bar{u}_2(\cdot)$ to maximize \eqref{eq:1.7}. Therefore, \eqref{eq:1.7}
represents the cost for Player I and the payoff for Player II.  If $(\bar{u}_1(\cdot),\bar{u}_2(\cdot))$ is an
optimal open-loop control, it is easy to check that
\begin{eqnarray}
J(\bar{u}_1(\cdot), {u}_2(\cdot))\leq J(\bar{u}_1(\cdot),
\bar{u}_2(\cdot))\leq J(u_1(\cdot), \bar{u}_2(\cdot))
\end{eqnarray}
for all admissible open-loop controls
$(u_1(\cdot),{u}_2(\cdot))\in {\cal A}_1\times{\cal A}_2$. We
refer to $(\bar{u}_1(\cdot), \bar{u}_2(\cdot))$ as an open-loop
saddle.  On the other hand, if $(\bar{u}_1(\cdot), \bar{u}_2(\cdot))$  is  an open-loop
saddle, by  Remark 1.2 of Chapter VI in [6], we know that

\begin{equation}\label{eq:2.9}
\sup_{u_2(\cdot)\in {\cal
A}_2^W}\left(\inf_{u_1(\cdot)\in {\cal
A}_1^W}J(u_1(\cdot),u_2(\cdot))\right)
=\inf_{
u_1(\cdot)\in {\cal
A}_1^W}\left(\inf_{u_2(\cdot)\in {\cal
A}_2^W}J(u_1(\cdot),u_2(\cdot))\right).
\end{equation}

In this paper,  provided  the  original
sate equation   \eqref{eq:1.1} and
the observation equation  \eqref{eq:1.2},
 we will also study the partially observed  optimal control
problem  in its strong formulation, i.e.
without changing the  reference
probability space $(\Omega, \mathscr F, \{\mathscr{F}_t\}_{0\leq t\leq T},\mathbb P).$
Precisely,  different from
the  cost functional \eqref{eq:1.6},
  the cost functional in this case  is defined by
\begin{equation}\label{eq:1.10}
\begin{split}
J(u_1(\cdot ), u_2(\cdot))=& \mathbb E\displaystyle\bigg[%
\int_{0}^{T}l(t,X(t),\mathbb E [ x(t)], u_1(t),
\mathbb E[u_1(t)], u_2(t), \mathbb E[u_2(t)])dt \\
& ~~~~~+m(X(T),\mathbb E [ x(T)])\bigg].
\end{split}
\end{equation}
Note that $\mathbb E(\cdot)$ is
the expectation with  the original
probability $\mathbb P$ independent
of the control $(u_1(\cdot), u_2(\cdot)) .$   In
this case,   different from the
partially observed  optimal control problem in weak sense
discussed before, we do not need
require the pair of admissible
control processes satisfies \eqref{eq:1.33}. In this case, a pair of   admissible
control processes $(u_1(\cdot), u_2(\cdot))$ is defined as a $\{\mathscr{F}^Y_t\}_{0\leq t\leq T}$
adapted stochastic process valued in $U_1\times U_2$
satisfying
\begin{equation} \label{eq:2.14}
  \mathbb E\bigg[\int_0^T|u_1(t)|^2
dt\bigg]+\mathbb E\bigg[\int_0^T|u_2(t)|^2
dt\bigg]<\infty.
\end{equation}
The set of all admissible controls in this case
is denoted by ${\cal A}_1^S\times {\cal A}_2^S.$

Then we  can put forward the partially observed two-person zero-sum differential game problem   in its strong formulation as follows.
\begin{pro}
\label{pro:4.1} Find an admissible open-loop control
$(\bar{u}_1(\cdot), \bar{u}_2(\cdot))$ over  ${\cal A}_1^W\times{\cal A}_2^W$
such that
\begin{equation}\label{eq:2.15}
J(\bar{u}_1(\cdot), \bar{u}_2(\cdot))=\sup_{u_2(\cdot)\in {\cal
A}_2^S}\left(\inf_{u_1(\cdot)\in {\cal
A}_1^S}J(u_1(\cdot),u_2(\cdot))\right)
=\inf_{
u_1(\cdot)\in {\cal
A}_1^S}\left(\inf_{u_2(\cdot)\in {\cal
A}_2^S}J(u_1(\cdot),u_2(\cdot))\right),
\end{equation}
 subject to
 the cost functional \eqref{eq:1.7}
 and  the following
 state equation
\begin{equation}
\displaystyle\left\{
\begin{array}{lll}
dx(t) & = & [(b-{\tilde
g}h)(t,x(t),
u_1(t),
\mathbb E[u_1(t)], u_2(t), \mathbb E[u_2(t)]]dt
\\&&+\displaystyle
g(t,x(t),\mathbb E [x(t)],u_1(t),
\mathbb E[u_1(t)], u_2(t), \mathbb E[u_2(t)])dW(t)
 \\&&+\displaystyle {\tilde g}(t,x(t),\mathbb E[ x(t)],u_1(t),
\mathbb E[u_1(t)], u_2(t), \mathbb E[u_2(t)])dY(t),
\\
\displaystyle x(0) & = & a\in \mathbb{R}^n.
\end{array}%
\right.  \label{eq:2.16}
\end{equation}
\end{pro}

Note that  under Assumptions\ref{ass:1.1}
and \ref{ass:1.2},
for any admissible control $({u}_1(\cdot), {u}_2(\cdot))
 \in {\cal A}_1^S\times {\cal A}_2^S$, 
by Lemma 1.4 in Tang and Meng (2016),
the state \eqref{eq:3.7} has a unique solution
$x(\cdot)\in S_{\mathscr F}^2(0,T; \mathbb R^n)$
and $J(u_1(\cdot), u_2(\cdot)) < \infty,$  so Problem \ref{pro:4.1} is well-defined.

\section{Stochastic Maximum Principle  For Zero-Sum Differential Games  in Weak Formulation }

\noindent

  This section is devoted to establishing
   the stochastic maximum principle of
   Problem \ref{pro:1.1} or   Problem \ref{pro:1.2},
   i.e., establishing the necessary optimality condition of Pontryagin's type for
   a pair of  admissible controles to be optimal. To this end,
for the state equation \eqref{eq:3.7},
we first introduce the
corresponding adjoint equation.
Actually, define
the Hamiltonian function $H:[0,T]\times \Omega \times
\mathbb{R}^n \times \mathbb{R}^n \times U_1 \times U_2\times U_2\times U_2
\times \mathbb{R}^n \times \mathbb{R}^n
\times \mathbb{R}^n \times \mathbb R
\longrightarrow \mathbb{R}$ by
\begin{eqnarray} \label{eq:2.1}
\begin{split}
  &H(t,x,y,u_1,v_1,u_2,v_2,p,q, \tilde {q}, \tilde R)
\\&=\langle p,  b(t,x,y,,u_1,v_1,u_2,v_2)\rangle +\langle q,  g(t,x,y,u_1,v_1,u_2,v_2)\rangle
+\langle  \tilde q, \tilde g(t,x,y,u_1,v_1,u_2,v_2)\rangle
\\&+ \tilde Rh(t,x,y, u_1,v_1,u_2,v_2)+l(t,x,y,u_1,v_1,u_2,v_2).
\end{split}
\end{eqnarray}
For the state equation \eqref{eq:3.7}
associated with any given admissible pair
$(\bar u_1(\cdot), \bar u_2(\cdot),\bar x(\cdot), \bar Z(\cdot)),$ the corresponding
adjoint equation is defined as follows:

\begin{numcases}{}\label{eq:3.10}
\begin{split}
d\bar r(t)&=-l(t,\bar x(t), \mathbb E [\bar x(t)], \bar u_1(t), \mathbb E [\bar u_1(t)],\bar u_2(t), \mathbb E [\bar u_2(t)])dt+
\bar R\left(
t\right)  dW\left(  t\right) +\bar {\tilde R}\left( t\right) dW^{\bar
u}\left( t\right),
\\
d\bar p\left(  t\right)  &=-\Big\{{H}_{x}\left(  t,\bar x(t), \mathbb E [\bar
x(t)], \bar u_1(t), \mathbb E [\bar u_1(t)],\bar u_2(t), \mathbb E [\bar u_2(t)]\right)
\\&\quad\quad+\frac{1}{z^{(u_1,u_2)}(t)}\mathbb E^{(u_1,u_2)}\big[{H}_{y}\left( t,\bar x(t), \mathbb E [\bar x(t)], \bar u_1(t), \mathbb E [\bar u_1(t)],\bar u_2(t), \mathbb E [\bar u_2(t)]\right)\big]  \Big\}dt
\\&~~~~~~+\bar q
\left(  t\right)  dW\left(  t\right)  +\bar{\tilde q}\left(  t\right)
dW^{(\bar u_1, \bar u_2) }\left( t\right),\\
\bar r(T)&=m(\bar x(T),\mathbb E[\bar x (T)]),
\\ \bar p(T)&=m_x(\bar x(T),\mathbb E[\bar x (T)])+\frac{1}{z^{(u_1,u_2)}(T)}\mathbb E^u \left[m_x(\bar x(T),\mathbb E [\bar x(T)])\right],
\end{split}
\end{numcases}
where
\begin{eqnarray} \label{eq:3.11}
\begin{split}
  &H(t,x,y,u_1,v_1, u_2,v_2)
  \\&=:{ H}(t,x, y,u_1,v_1, u_2,v_2, \bar p(t),
  \bar q(t), \bar {\tilde q}(t),\bar{\tilde R}(t)-\tilde g(t,\bar x(t), \mathbb E
[\bar x(t)], \bar u_1(t), \mathbb E [\bar u_1(t)],
\bar u_2(t), \mathbb E [\bar u_2(t)] )^\top \bar p(t)).
\end{split}
\end{eqnarray}
Note the adjoint equation \eqref{eq:3.10}
is a mean-field backward
stochastic differential equation whose solution
consists of  an 6-tuple process $(\bar p(\cdot),\bar q(\cdot),\bar {\tilde q}(\cdot),\bar r(\cdot),\bar R(\cdot),\bar {\tilde R}(\cdot ) ).$
Under Assumptions \ref{ass:1.1} and
\ref{ass:1.2},  by Buckdahn (2009b),
it is easily to see that  the adjoint equation \eqref{eq:3.10} admits
a unique solution $(\bar p(\cdot),\bar q(\cdot),\bar {\tilde q}(\cdot),\bar r(\cdot),\bar R(\cdot),\bar {\tilde R}(\cdot) )\in S_{\mathscr{F}}^2(0,T;
\mathbb R^n)\times M_{\mathscr{F}}^2(0,T;
\mathbb R^n) \times M_{\mathscr{F}}^2(0,T;
\mathbb R^n)\times S_{\mathscr{F}}^2(0,T;
\mathbb R)\times M_{\mathscr{F}}^2(0,T;
\mathbb R) \times M_{\mathscr{F}}^2(0,T;
\mathbb R),
$  also called the adjoint process corresponding
the admissible pair $(\bar{u}(\cdot ); \bar{x}(\cdot ), \bar Z(\cdot))$.

Now we are in a position to
state our main result:  stochastic
maximum principle of  Problem \ref{pro:1.1}  or \ref{pro:1.2}.

\begin{thm}
 Let assumptions \ref{ass:1.1} and \ref{ass:1.2} be satisfied.  Let $(\bar{u}_1(\cdot), \bar{u}_2(\cdot))$
  be an optimal open-loop control  of  Problem \ref{pro:1.1} or \ref{pro:1.2}. Suppose that
    $(\bar{x}(\cdot ), \bar Z(\cdot))$ is
 the state process of the system (\ref{eq:3.7}) corresponding to $(\bar{u}_1(\cdot), \bar{u}_2(\cdot)).$
 Let $(\bar p(\cdot),\bar q(\cdot),\bar {\tilde q}(\cdot),\bar r(\cdot),\bar R(\cdot),\bar {\tilde R}(\cdot) )$ be  the unique
 solution of the  adjoint equation (\ref{eq:3.10}) corresponding $
 (\bar{u}_1(\cdot), \bar{u}_2(\cdot);  \bar{x}(\cdot ), \bar Z(\cdot))$.
 Then
the optimality conditions
\begin{eqnarray}\label{eq:3.4}
\begin{split}
  &\Big\langle\mathbb E \big[ Z^{(\bar u_1, \bar u_2)}(t){\bar H}_{u_1}(t)|\mathscr F_t^Y]
+\mathbb E^{(\bar u_1, \bar u_2)}[{ \bar H}_{v_1}(t)], u_1-\bar u_1(t) \Big\rangle\geq 0
\end{split}
\end{eqnarray}
and

\begin{eqnarray}\label{eq:3.5}
\begin{split}
  &\Big\langle\mathbb E \big[ Z^{(\bar u_1, \bar u_2)}(t){\bar H}_{u_2}(t)|\mathscr F_t^Y]
+\mathbb E^{(\bar u_1, \bar u_2)}[{ \bar H}_{v_2}(t)], u_2-\bar u_2(t) \Big\rangle\leq 0
\end{split}
\end{eqnarray}
hold for any $(u_1, u_2)\in U_1\times U_2$  and a.e. $(t,\omega )\in \lbrack 0,T]\times \Omega .$ Here
using the notation \eqref{eq:3.11},  for $i=1,2,$ we set
\begin{equation} \label{eq:3.13}
  \bar H_{u_i}(t)={ H}_u(t, \bar x(t), \mathbb E[\bar \bar x(t)],
  \bar u_1(t),\mathbb E[\bar u_1(t)], \bar u_2(t), \mathbb E[\bar u_2(t)])
\end{equation}
and
\begin{equation}\label{eq:3.14}
  \bar H_{v_i}(t)={ H}_{v_i}(t, x(t), \mathbb E[\bar x(t)],
  u_1(t),\mathbb E[u_1(t)], u_2(t), \mathbb E[u_2(t)] ).
\end{equation}
 \end{thm}

 {\bf Proof:} Since  $(\bar{u}_1(\cdot),\bar{u}_2(\cdot))$ is an
optimal open-loop control, then  $(\bar{u}_1(\cdot),\bar{u}_2(\cdot))$
is an open-loop
saddle point, i.e.,
\begin{eqnarray}
J(\bar{u}_1(\cdot), {u}_2(\cdot))\leq J(\bar{u}_1(\cdot),
\bar{u}_2(\cdot))\leq J(u_1(\cdot), \bar{u}_2(\cdot)).
\end{eqnarray}
So we have \begin{equation}\label{eq:3.9}
    J_1(\bar{u}_1(\cdot),
\bar{u}_2(\cdot))=\displaystyle\min_{u_1(\cdot)\in {\cal
A}_1^W}J_1(u_1(\cdot), \bar{u}_2(\cdot)) \end{equation}
and
\begin{equation}\label{eq:3.100}
 J_2(\bar{u}_1(\cdot),
\bar{u}_2(\cdot))=\displaystyle\max_{u_2(\cdot)\in {\cal
A}_2^W}J_2(\bar u_1(\cdot), {u}_2(\cdot)).
\end{equation}
By (\ref{eq:3.9}),  $\bar u_1(\cdot)$  can be regarded as
an optimal control  of an
optimal control problem where
the controlled system is

\begin{equation}
\displaystyle\left\{
\begin{array}{lll}
dx(t) & = & [(b-{\tilde
g}h)(t,x(t),
u_1(t),
\mathbb E[u_1(t)], \bar u_2(t), \mathbb E[\bar u_2(t)]]dt
\\&&+\displaystyle
g(t,x(t),\mathbb E [x(t)],u_1(t),
\mathbb E[u_1(t)], \bar u_2(t), \mathbb E[\bar u_2(t)])dW(t)
 \\&&+\displaystyle {\tilde g}(t,x(t),\mathbb E[ x(t)],u_1(t),
\mathbb E[u_1(t)], \bar u_2(t), \mathbb E[\bar u_2(t)])dY(t),
 \\
dZ(t) & = & Z(t)h(t,x(t),\mathbb E [ x(t)], u_1(t),
\mathbb E[u_1(t)], \bar u_2(t), \mathbb E[\bar u_2(t)])dY(t), \\
\displaystyle Z(0) & = & 1,%
\\
\displaystyle x(0) & = & a\in \mathbb{R}^n,
\end{array}%
\right.  \label{eq:3.111}
\end{equation}

 and the cost functional
is \begin{equation}\label{eq:3.12}
\begin{split}
J(u_1(\cdot ),\bar u_2(\cdot))=& \mathbb E\displaystyle\bigg[%
\int_{0}^{T}Z^{(u_1, \bar u_2)}(t)l(t,x(t),\mathbb E [ x(t)], u_1(t),
\mathbb E[u_1(t)], \bar u_2(t), \mathbb E[\bar u_2(t)])dt \\
& ~~~~~+Z^{(u_1,\bar u_2)}(T)m(x(T),\mathbb E [ x(T)])\bigg].
\end{split}
\end{equation}
Then for this case, it is easy to check that the Hamilton is  $H(t,x,y,u_1,v_1,\bar u_2(t),
 \mathbb E [\bar u_2(t)],p,q, \tilde {q}, \tilde R)$
and  for the optimal control
$\bar u_1(\cdot)\in {\cal A}_1^W,$   the corresponding optimal  sate process and the adjoint process
is still $\bar{x}(\cdot)$ and $(\bar p(\cdot),\bar q(\cdot),\bar {\tilde q}(\cdot),\bar r(\cdot),\bar R(\cdot),\bar {\tilde R}(\cdot) ),$  respectively.
Thus applying the     partial
necessary stochastic maximum principle for  optimal control problems
(see Theorem 2.1 in  Tang and Meng (2016)), we can obtain \eqref{eq:3.4}.
Similarly, from (\ref{eq:3.100}), we can obtain (\ref{eq:3.5}). The proof is complete.

\section{ Stochastic Maximum Principle For Zero-Sum Differential Games in Strong  Formulation }

This section is devoted to establishing
the stochastic maximum principles of
Problem \ref{pro:4.1}.  In this case, the Hamiltonian
$H:[0,T]\times \Omega \times
\mathbb{R}^n \times
\mathbb{R}^n\times U_1\times U_1\times U_2\times U_2\times
\mathbb{R}^n\times
\mathbb{R}^n\times
\mathbb{R}^n  %
\rightarrow \mathbb{R}$ is defined  by
\begin{eqnarray} \label{eq:4.3}
\begin{split}
  H(t,x,y,u_1,v_1,u_2,v_2,p,q, \tilde {q})=&\langle p,  b(t,x,y,u_1,v_1,u_2,v_2)-\tilde g(t,x,y,u_1,v_1, u_2,v_2)h(t, x,y, u_1,v_1,u_2,v_2)\rangle
  \\&+\langle q,
g(t,x,y,u_1,v_1,u_2,v_2)\rangle +   \langle\tilde q,\tilde g(t,x,y,u_1,v_1,u_2,v_2) \rangle \\&\quad+l(t,x,y,u_1,v_1,u_2,v_2).
\end{split}
\end{eqnarray}
Then for any admissible pair
$(\bar u_1(\cdot), \bar u_2(\cdot); \bar x(\cdot)),$  the
corresponding  adjoint process
 is  defined   as   the solution to  the following mean-field
BSDE:

\begin{numcases}{}\label{eq:4.4}
\begin{split}
d\bar p\left(  t\right)  &=-\bigg[{\bar H}_{x}(t)+\mathbb E[\bar H_y(t)] \bigg ]dt+\bar q\left(  t\right)  dW\left(  t\right)  +\bar{\tilde q}\left(  t\right)
dY\left( t\right)
\\ \bar P(T)&= \bar m_x( T)+\mathbb E[ \bar m_{y}(T)],
\end{split}
\end{numcases}
where we have used the following
shorthand notation
\begin{eqnarray}\label{eq:3.3}
\left\{
\begin{aligned}
    \bar H(t)=&{ H}(t,\bar x(t), \mathbb E[\bar x(t)] ,\bar u_1(t),
  \mathbb E[\bar u_1(t)],\bar u_2(t),
  \mathbb E[\bar u_2(t)],\bar p(t),\bar q(t),
   {\bar{\tilde q}}(t)),\\
   \bar m(T)=& m(\bar x(T),  \mathbb E[\bar x(T)]).
\end{aligned}
\right.
\end{eqnarray}
Under Assumption \ref{ass:1.1} and
\ref{ass:1.2},
by Buckdahn (2009),  \eqref{eq:4.4} admits
a unique strong  slution $(\bar p(\cdot), \bar q(\cdot),
\bar {\tilde q}(\cdot))\in S_{\mathscr{F}}^2 (0,T;\mathbb R^{n})\times M_{\mathscr{F}}^2(0,T;\mathbb R^n)
\times M_{\mathscr{F}}^2(0,T;\mathbb R^n),$
which is also called the adjoint process corresponding
to the admissible pair $(u_1(\cdot), u_2(\cdot), x(\cdot))$

\subsection{ Sufficient Conditions of Optimality }
In this subsection, we are going to establish
the sufficient Pontryagin maximum principle of Problem  \ref{pro:4.1}.

\begin{thm}{\bf [Sufficient Stochastic Maximum Principle ] } \label{thm:4.3}
Let Assumptions \ref{ass:1.1} and
\ref{ass:1.2} be
satisfied. Let $(\bar{u}_1(\cdot), \bar{u}_2(\cdot);
\bar{x}(\cdot))$ be an admissible pair
and $(\bar p(\cdot), \bar q(\cdot), \bar{\tilde q}(\cdot))$ be  the  unique strong solution of the
corresponding adjoint  equation (\ref{eq:4.4}).\\
(i) Suppose that, for all $t\in [0,T],$  $m(x,y)$ is convex in $(x,y)$,  and the mapping
$$
(x,y,u_1,v_1)\mapsto
H( t, x, y, u_1,v_1, \bar u_2(t),\mathbb E[\bar u_2(t)], {\bar p} (t), {\bar q} (t), \bar {\tilde q}(t) )
$$ is convex.  For  any $u_1(\cdot)\in {\cal A}_1^S,$
\begin{eqnarray} \label{eq:4.8}
  \mathbb E\bigg[\langle  u_1 (t) - \bar u_1 (t), \bar H_{u_1}(t)+\mathbb E [\bar H_{v_1}(t)]\rangle\bigg]
\geq 0.
\end{eqnarray}
Then
$$
J(u_1(\cdot),\bar{u}_2(\cdot)) \geq J(\bar{u}_1(\cdot),\bar{u}_2(\cdot))
,~~~for~~all ~~u_1(\cdot)\in{\cal
A}_1^S
$$
and
$$
J(\bar{u}_1(\cdot),\bar{u}_2(\cdot))=\displaystyle\inf_{u_1(\cdot)\in{\cal
A}_1^S}J(u_1(\cdot),\bar{u}_2(\cdot)).
$$
(ii) Suppose that, for all $t\in [0,T],$ $m(x,y)$ is concave in $(x,y)$  and
$$
(x,y,u_2,v_2)\mapsto
H( t, x, y, \bar u_1(t), \mathbb E[\bar u_1(t)], u_2,v_2, {\bar p} (t), {\bar q} (t), \bar {\tilde q}(t) )
$$
is concave. For  any $u_2(\cdot)\in {\cal A}_2^S,$
\begin{eqnarray} \label{eq:4.80}
  \mathbb E\bigg[\langle  u_2 (t) - \bar u_2 (t), \bar H_{u_2}(t)+\mathbb E [\bar H_{v_2}(t)]\rangle\bigg]
\leq 0.
\end{eqnarray}      Then  $$ J(\bar{u}_1(\cdot),{u}_2(\cdot))\leq J(\bar{u}_1(\cdot),\bar{u}_2(\cdot))
,~~~for~~all ~~u_2(\cdot)\in{\cal
A}_2^S
$$
and
$$
J(\bar{u}_1(\cdot),\bar{u}_2(\cdot))
=\displaystyle\sup_{u_2(\cdot)\in{\cal
A}_2^S}J(\bar{u}_1(\cdot),u_2(\cdot)).
$$

(iii) If both case (i) and (ii) hold (which implies, in particular,
that $m(x,y)$  is an affine function), then
$(\bar{u}_1(\cdot),\bar{u}_2(\cdot))$ is an open-loop saddle point
and
\begin {equation}
\begin{array}{ll}
J(\bar{u}_1(\cdot), \bar{u}_2(\cdot))=\sup_{u_2(\cdot)\in {\cal
A}_2^S}\left(\inf_{u_1(\cdot)\in {\cal
A}_1^S}J(u_1(\cdot),u_2(\cdot))\right)
=\inf_{
u_1(\cdot)\in {\cal
A}_1^S}\left(\inf_{u_2(\cdot)\in {\cal
A}_2^S}J(u_1(\cdot),u_2(\cdot))\right).
\end{array}
\end {equation}
\end{thm}

  \begin{proof}(i) In the following,  we  consider a stochastic optimal control problem over ${\cal A}_1^S$ where the system  is
\begin{equation}
\displaystyle\left\{
\begin{array}{lll}
dx(t) & = & [(b-{\tilde
g}h)(t,x(t),
u_1(t),
\mathbb E[u_1(t)], \bar u_2(t), \mathbb E[\bar u_2(t)]]dt
\\&&+\displaystyle
g(t,x(t),\mathbb E [x(t)], u_1(t),
\mathbb E[u_1(t)], \bar u_2(t), \mathbb E[\bar u_2(t)])dW(t)
 \\&&+\displaystyle {\tilde g}(t,x(t),\mathbb E[ x(t)],u_1(t),
\mathbb E[u_1(t)], \bar u_2(t), \mathbb E[\bar u_2(t)])dY(t),
\\
\displaystyle x(0) & = & a\in \mathbb{R}^n
\end{array}%
\right.  \label{eq:4.7}
\end{equation}
with the cost functional
\begin{equation}\label{eq:4.8}
\begin{split}
J(u_1(\cdot ), \bar u_2(\cdot))=& \mathbb E\displaystyle\bigg[%
\int_{0}^{T}l(t,X(t),\mathbb E [ x(t)], u_1(t),
\mathbb E[u_1(t)], \bar u_2(t), \mathbb E[\bar u_2(t)])dt \\
& ~~~~~+m(X(T),\mathbb E [ x(T)])\bigg].
\end{split}
\end{equation}
Our  optimal control problem is minimize
$J(u_1(\cdot), \bar{u}_2(\cdot))$ over $u_1(\cdot) \in {\cal A }_1^S$,
i.e., find  $\bar{u}_1(\cdot)\in {\cal A }_1^S$ such
that\begin{equation} J(\bar{u}_1(\cdot),
\bar{u}_2(\cdot))=\displaystyle\inf_{u_1(\cdot)\in {\cal
A}_1}J(u_1(\cdot), \bar{u}_2(\cdot)).
\end{equation}
Then for this case, it is easy to check that the Hamilton is  $H(t,x,y,u_1,v_1,\bar u_2(t),
\mathbb E[\bar u_2(t)],p,q, \tilde {q})$
and  for the admissible control
$\bar u_1(\cdot)\in {\cal A}_1^S,$   the corresponding  sate process and the adjoint process
is still $\bar{x}(\cdot)$ and $(\bar p(\cdot), \bar q(\cdot), \bar{\tilde q}(\cdot))$  respectively.
Thus from the partial observed sufficient maximum principle for optimal control (see Theorem
 3.1 in Tang and Meng (2016)), we conclude that $\bar{u}_1(\cdot)$ is  the optimal
 control of the optimal control problem, i.e.,
$$
J(\bar{u}_1(\cdot),\bar{u}_2(\cdot))\leq
J(u_1(\cdot),\bar{u}_2(\cdot)),~~~for~~all ~~u_1(\cdot)\in{\cal
A}_1^S,
$$
and
$$
J(\bar{u}_1(\cdot),\bar{u}_2(\cdot))=\displaystyle\inf_{u_1(\cdot)\in{\cal
A}_1^S}J(u_1(\cdot),\bar{u}_2(\cdot)).
$$
     The proof of (i) is complete.\\
(ii) This statement can be proved in a similar way as (i).\\
(iii) if both (i) and (ii) hold, then
$$
J(\bar{u}_1(\cdot),u_2(\cdot))\leq
J(\bar{u}_1(\cdot),\bar{u}_2(\cdot))\leq
J(u_1(\cdot),\bar{u}_2(\cdot)),
$$
for any $(u_1(\cdot),u_2(\cdot))\in {\cal A}_1^S\times {\cal A}_2^S.$
Thereby,
$$
\begin{array}{ll}
J(\bar{u}_1(\cdot),\bar{u}_2(\cdot))&\leq
\displaystyle\inf_{u_1(\cdot)\in {\cal
A}_1^S}J(u_1(\cdot),\bar{u}_2(\cdot))\\
&\leq \sup_{u_2(\cdot)\in {\cal
A}_2^S}\left(\displaystyle\inf_{u_1(\cdot)\in {\cal
A}_1^S}J(u_1(\cdot),u_2(\cdot))\right).
\end{array}
$$
On the other hand,
$$
\begin{array}{ll}
J(\bar{u}_1(\cdot),\bar{u}_2(\cdot))&\geq
\displaystyle\sup_{u_2(\cdot)\in {\cal
A}_2^S}J(\bar{u}_1(\cdot),{u}_2(\cdot))\\
&\geq \displaystyle\inf_{u_1(\cdot)\in {\cal
A}_1^S}\left(\displaystyle\sup_{u_2(\cdot)\in {\cal
A}_2^S}J(u_1(\cdot),u_2(\cdot))\right).
\end{array}
$$
Now, due to the inequality
$$
\begin{array}{ll}
&\displaystyle\inf_{u_1(\cdot)\in {\cal A}_1^S}\left(\sup_{u_2(\cdot)\in
{\cal A}_2^S}J(u_1(\cdot),u_2(\cdot))\right)\\
\geq &\sup_{u_2(\cdot)\in {\cal A}_2^S}\left(\inf_{u_1(\cdot)\in {\cal
A}_1^S}J(u_1(\cdot),u_2(\cdot))\right),
\end{array}
$$
we have
$$
\begin{array}{ll}
J(\bar{u}_1(\cdot),\bar{u}_2(\cdot))
&=\displaystyle\sup_{u_2\in {\cal A}_2^S}\left
(\inf_{u_1\in {\cal A}^S_1}J(u_1,u_2)\right )\\
&=\inf_{u_1(\cdot)\in{\cal A}_1^S}\left(\sup_{u_2(\cdot)\in {\cal A}_2^S}J(u_1,u_2)
\right).
\end{array}
$$
The proof  is completed.
\end{proof}

\subsection{ Necessary Conditions of Optimality}

In this subsection we  give
the necessary Pontryagin maximum principle of Problem  \ref{pro:4.1}.
\begin{thm}\label{thm:4.5}
 Let  Assumption \ref{ass:1.1} and \ref{ass:1.2} be satisfied.
   Let $(\bar{u}_1(\cdot), \bar{u}_2(\cdot))$
  be an optimal open-loop control of  Problem \ref{pro:4.1}. Suppose that
    $\bar{x}(\cdot )$ is
 the state process of the system (\ref{eq:3.7}) corresponding to
 the admissible control $(\bar{u}_1(\cdot), \bar{u}_2(\cdot)).$
 Let $(\bar p(\cdot),\bar q(\cdot),\bar {\tilde q}(\cdot),\bar r(\cdot),\bar R(\cdot),\bar {\tilde R}(\cdot) )$ be  the unique
 solution of the  adjoint equation (\ref{eq:4.4}) corresponding to $
 (\bar{u}_1(\cdot), \bar{u}_2(\cdot);  \bar{x}(\cdot ))$.
 Then
the optimality conditions
\begin{eqnarray}\label{eq:4.10}
\begin{split}
  &\Big\langle\mathbb E \big[ {\bar H}_{u_1}(t)|\mathscr F_t^Y]
+\mathbb E[{ \bar H}_{v_1}(t)], u_1-\bar u_1(t) \Big\rangle\geq 0
\end{split}
\end{eqnarray}
and

\begin{eqnarray}\label{eq:4.11}
\begin{split}
  &\Big\langle\mathbb E \big[ {\bar H}_{u_2}(t)|\mathscr F_t^Y]
+\mathbb E[{ \bar H}_{v_1}(t)], u_2-\bar u_2(t) \Big\rangle\leq 0
\end{split}
\end{eqnarray}
holds for any $(u_1, u_2)\in U_1\times U_2$  and a.e. $(t,\omega )\in \lbrack 0,T]\times \Omega .$ Here
using the notation \eqref{eq:3.111},  for $i=1,2,$ we set
\begin{equation} \label{eq:3.13}
  \bar H_{u_i}(t)={ H}_{u_i}(t, \bar x(t), \mathbb E[\bar x(t)],
  \bar u_1(t),\mathbb E[\bar u_1(t)], \bar u_2(t), \mathbb E[\bar u_2(t)])
\end{equation}
and
\begin{equation}\label{eq:3.14}
  \bar H_{v_i}(t)={ H}_{v_i}(t, \bar x(t), \mathbb E[\bar x(t)],
  \bar u_1(t),\mathbb E[\bar u_1(t)], \bar u_2(t), \mathbb E[\bar u_2(t)] ).
\end{equation}
\end{thm}

 {\bf Proof:} Since  $(\bar{u}_1(\cdot),\bar{u}_2(\cdot))$ is an
optimal open-loop control, then  $(\bar{u}_1(\cdot),\bar{u}_2(\cdot))$
is an open-loop
saddle point, i.e.,
\begin{eqnarray}
J(\bar{u}_1(\cdot), {u}_2(\cdot))\leq J(\bar{u}_1(\cdot),
\bar{u}_2(\cdot))\leq J(u_1(\cdot), \bar{u}_2(\cdot)).
\end{eqnarray}
So we have \begin{equation}\label{eq:3.3}
    J_1(\bar{u}_1(\cdot),
\bar{u}_2(\cdot))=\displaystyle\min_{u_1(\cdot)\in {\cal
A}_1^S}J_1(u_1(\cdot), \bar{u}_2(\cdot)) \end{equation}
and
\begin{equation}\label{eq:4.161}
 J_2(\bar{u}_1(\cdot),
\bar{u}_2(\cdot))=\displaystyle\max_{u_2(\cdot)\in {\cal
A}_2^S}J_2(\bar u_1(\cdot), {u}_2(\cdot)).
\end{equation}
By (\ref{eq:3.3}),  $\bar u_1(\cdot)$  can be regarded as
an optimal control  of the
optimal control problem where
the controlled system is
(\ref{eq:4.7}) and the cost functional
is (\ref{eq:4.8}).
Then for this case, it is easy to check that the Hamilton is  $H(t,x,y,u_1,v_1,\bar u_2(t),
\mathbb E[\bar u_2(t)],p,q, \tilde {q})$
and  for the admissible control
$\bar u_1(\cdot)\in {\cal A}_1^S,$   the corresponding  sate process and the adjoint process
is still $\bar{x}(\cdot)$ and $(\bar p(\cdot), \bar q(\cdot), \bar{\tilde q}(\cdot))$  respectively.
Thus applying the     partial
necessary stochastic maximum principle for  optimal control problems
(see Theorem 3.5 in Tang and Meng (2016)), we can obtain \eqref{eq:4.10}.
Similarly, from (\ref{eq:4.161}), we can obtain (\ref{eq:4.11}). The proof is completed.

\section {An Example:  A Linear Quadratic Differential Game Problem}

In this section, we apply
our  stochastic maximum  principle to
solve  a  partial observed  stochastic  linear quadratic differential game  problem. Let us make it more precise below.
In this case, we assume the  state system is the
following linear mean-field  SDE
\begin{equation}\label{eq:5.1}
\left\{\begin {array}{ll}
  dX(t)=&(A_1(t)X(t)+A_2(t)\mathbb E [X(t)]
  +B_{11}(t)u_1(t)+B_{12}(t)\mathbb E [u_1(t)]
  +B_{21}(t)u_2(t)+B_{22}(t)\mathbb E [u_2(t)])dt
  \\&+(C_1(t)X(t)
  + C_2(t)\mathbb E [X(t)]
  +D_{11}(t)u_1(t)+D_{12}(t)\mathbb E [u_1(t)]
  +D_{21}(t)u_2(t)+D_{22}(t)\mathbb E [u_2(t)])dW(t)
  \\&+(F_1(t)X(t)
  +F_2(t)\mathbb E [X(t)]
  +G_{11}(t)u_1(t)+G_{12}(t)\mathbb E [u_1(t)]
  +G_{21}(t)u_2(t)+G_{22}(t)\mathbb E [u_2(t)])dW^{(u_1,u_2)}(t),
   \\x(0)=&x \in \mathbb R^n,
\end {array}
\right.
\end{equation}
with an observation

\begin{equation}\label{eq:5.2}
\displaystyle\left\{
\begin{array}{lll}
dY(t) & = & h(t)dt+dW^{(u_1,u_2)}(t),  \\_{}
\displaystyle Y(0) & = & 0,%
\end{array}
\right.
\end{equation}
and the cost functional  has the following
quadratic form:

 \begin{eqnarray}\label{eq:5.3}
\begin{split}
J ( u_1 (\cdot), u_2(\cdot) ) = &{\mathbb E} [
 \langle M, X(T) \rangle ]+{\mathbb E} \bigg [ \int_0^T \langle Q (s), X
(s) \rangle d s \bigg ]
\\&+ {\mathbb E} \bigg [ \int_0^T
 \langle N_{11} (s) u_1 (s), u_1 (s) \rangle d s\bigg]+ {\mathbb E} \bigg [ \int_0^T \langle N_{12} (s)
\mathbb E[u_1 (s)], \mathbb E[u_1 (s)] \rangle d s
\bigg ]
\\&+ {\mathbb E} \bigg [ \int_0^T
 \langle N_{21} (s) u_2 (s), u_2 (s) \rangle d s\bigg]+ {\mathbb E} \bigg [ \int_0^T \langle N_{22} (s)
\mathbb E[u_{2} (s)], \mathbb E[u_2 (s)] \rangle d s
\bigg ].
\end{split}
\end{eqnarray}
 In this case, our  control process  $(u_1(\cdot), u_2(\cdot))$  is said to be  an admissible  stochastic
process if  $(u_1(\cdot), u_2(\cdot))\in
M_{{\mathscr F}^{Y}}^2(0,T; \mathbb R^{k_1})\times M_{{\mathscr F}^{Y}}^2(0,T; \mathbb R^{k_2}).$  The set of all admissible controls
is also denoted by ${\cal A}_1^S\times {\cal A}_2^S.$ Note that
there is no constraint on our control process,
since it takes value in  $\mathbb R^{k_1}\times \mathbb R^{K_2}.$  Now
we make the basic assumptions on
the coefficients.

\begin{ass}\label{ass:5.1}
 The matrix-valued functions $A_1, A_2, C_1,  C_2, F_1,  F_2,:[0, T]\rightarrow \mathbb R^{n\times n};
B_{11}, B_{12},D_{11}, D_{12},  G_{11},  G_{12},:[0, T]\rightarrow \mathbb R^{n\times k_1}
;
B_{21}, B_{22},D_{21}, D_{22},  G_{21},  G_{22},:[0, T]\rightarrow \mathbb R^{n\times k_2}
; N_{11}, N_{12}:[0, T]\rightarrow \mathbb R^{k_1\times k_1}; N_{21}, N_{22}:[0, T]\rightarrow \mathbb R^{k_2\times k_2}; Q:[0,T] \rightarrow \mathbb R^n;     h:[0, T]\rightarrow \mathbb R$ are uniformly bounded measurable functions.
$M$ is a vector in $\mathbb R^{n}.$
\end{ass}

\begin{ass}\label{ass:5.2}
 The matrix-valued functions $ N_{11} N_{11}+N_{12}$ are uniformly positive, i.e. for $\forall u_1\in \mathbb R^{k_1}$ and a.s. $t\in [0, T]$,
$ \langle N_{11}(t)u_1, u_1 \rangle \geq \delta \langle u_1, u_1\rangle
$ and $ \langle (N_{11}(t)+ N_{12}(t))u_1, u_1 \rangle \geq \delta \langle  u_1, u_1\rangle,
$  for some positive constant
$\delta$.    The matrix-valued functions $ N_{21} N_{21}+N_{22}$ are uniformly negative, i.e. for $\forall u_2 \in \mathbb R^{k_2}$ and a.s. $t\in [0, T]$,
$ \langle N_{21}(t)u_2, u_2 \rangle \leq -\delta \langle u_2, u_2\rangle
$ and $ \langle (N_{21}(t)+ N_{22}(t))u_2, u_2 \rangle \leq -\delta \langle  u_2, u_2\rangle,
$  for some positive constant
$\delta$.
\end{ass}
Our partial observed  stochastic  linear quadratic differential game control problem can be stated as follows.

 \begin{pro}\label{pro:5.1}

Find an admissible open-loop control
$(\bar{u}_1(\cdot), \bar{u}_2(\cdot))$ over  ${\cal A}_1^S\times{\cal A}_2^S$
such that
\begin{equation}\label{c1}
J(\bar{u}_1(\cdot), \bar{u}_2(\cdot))=\sup_{u_2(\cdot)\in {\cal
A}_2^S}\left(\inf_{u_1(\cdot)\in {\cal
A}_1^S}J(u_1(\cdot),u_2(\cdot))\right)
=\inf_{u_1(\cdot)\in {\cal
A}_1^S}\left(\inf_{u_2(\cdot)\in {\cal
A}_2^S}J(u_1(\cdot),u_2(\cdot))\right),
\end{equation}
   subject to \eqref{eq:5.1},
   \eqref{eq:5.2} and \eqref{eq:5.3}.
 \end{pro}

It is easy to check that  under
Assumptions \ref{ass:5.1} and \ref{ass:5.2},
if we set

\begin{eqnarray}\label{eq:5.5}
  \begin{split}
    &b(t,x,y,u_1,v_1, u_2,v_2)=A_1(t)x+A_2(t)y
  +B_{11}(t)u_1+B_{12}(t)v_1+B_{21}(t)u_2+B_{22}(t)v_2,
   \\& g(t,x,y,u_1,v_1,u_2,v_2)=C_1(t)x+C_2(t)y
  +D_{11}(t)u_1+D_{12}(t)v_1+D_{21}(t)u_2+D_{22}(t)v_2,
  \\&\tilde g(t,x,y,u,v)=
  F_1(t)x+F_2(t)y
  +G_{11}(t)u_1+G_{12}(t)v_1
  +G_{21}(t)u_2+G_{22}(t)v_2,
  \\&
  h(t,x,y,u_1,v_1, u_2,v_2)=h(t),
  m(x,y)=\langle M, x \rangle,
  \\&l(t,x,y,u_1,v_1, u_2, v_2)=\langle Q_1(t), x\rangle +\langle N_{11}(t)u_1, u_1\rangle +\langle N_{12}(t)v_1, v_1\rangle+\langle N_{21}(t)u_2, u_2\rangle +\langle N_{22}(t)v_2, v_2\rangle.
  \end{split}
\end{eqnarray}
Problem \ref{pro:5.1} can be regarded as  a special
case of Problem \ref{pro:4.1}
 and  Assumptions \ref{ass:1.1} and \ref{ass:1.2}  for \eqref{eq:5.5} hold. Thus  Theorem \ref{thm:4.3}
and \ref{thm:4.5} can be applied to solve  Problem \ref{pro:5.1}.  In this case,  the Hamiltonian becomes
\begin{eqnarray} \label{eq:5.6}
\begin{split}
 & H(t,x,y,u,v,p,q,\tilde q)
  \\= &
  \langle  p, A_1(t)x+A_2(t)y
  +B_{11}(t)u_1+B_{12}(t)v_1+B_{21}(t)u_2+B_{22}(t)v_2
  \\&\quad\quad- h(t)(F_1(t)x+F_2(t)y
  +G_{11}(t)u_1+G_{12}(t)v_1+G_{21}(t)u_2+G_{22}(t)v_2 ) \rangle
  \\&
    +\langle  q, C_1(t)x+C_2(t)y
  +D_{11}(t)u_1+D_{12}(t)v_1+D_{21}(t)u_2+D_{22}(t)v_2 \rangle
  \\&+\langle  \tilde q, F_1(t)x+F_2(t)y
  +G_{11}(t)u_1+G_{12}(t)v_1+G_{21}(t)u_2+G_{22}(t)v_2 \rangle
  \\& +\langle Q_1, x\rangle +
  \langle N_{11}u_1, u_1\rangle +\langle N_{12}v_1, v_1\rangle+
  \langle N_{21}u_2, u_2\rangle +\langle N_{22}v_2, v_2\rangle.
  \end{split}
\end{eqnarray}
For any admissible pair $(u_1(\cdot), u_2(\cdot); x(\cdot)),$
the corresponding adjoint equation becomes

 \begin {equation}\label{eq:5.7}
\left\{\begin{array}{lll}
dp(t)&=&-\bigg[(A^\top_1(t)-h(t)F_1^\top(t))p(t)+ (A_2^\top(t)-h(t)F_2^\top(t))\mathbb E [p(t)]+C^{\top}_1(t)q(t)
\\&&+ C^{\top}_2(t)\mathbb E[q(t)]
+ F^{\top}_1(t)\tilde q(t)+F^{\top}_2(t)\mathbb E[\tilde q(t)]+Q(t)\bigg]dt
\\&&+q(t)dW(t)+\tilde q(t)dY(t),
 \\p(T)&=&M.
\end{array}
  \right.
  \end {equation}
  In the following, we will prove  Problem \ref{pro:5.1} admits at least
  an open-loop saddle point $(\bar u_(
  \cdot),\bar  u_2(\cdot))\in {\cal A}_1^S\times
  {\cal A}_2^S.$  To the end, we need the  following basic result on  a saddle point in convex analysis  theory.

  \begin{lem} \label{lem:5.2} (Proposition 2.4 of Chapter VI in Ekeland and T\'{e}mam (1976)). Let $\cal A$ and $\cal B$ be two
    reflex  Banach space. Assume that a function
    $L $ of ${\cal A}\times {\cal B}\longrightarrow \mathbb R$ satisfies
    $$ \forall u\in {\cal A}, p\longrightarrow L(u,p)$$ is concave upper-semi continuous
    and
    $$ \forall p\in {\cal A}, u\longrightarrow L(u,p)$$ is convex lower-semi continuous.
    Moreover, there exist $u_0\in \cal A$ and
    $p_0\in \cal B$ such that
    $$ \lim_{||u||\longrightarrow\infty }L (u,p_0)=+\infty$$
    and
    $$ \lim_{||p||\longrightarrow\infty }L (u_0,p)=-\infty.$$
    Then $L$ possesses at least one  saddle point on $L.$

  \end{lem}

\begin{thm}
  Let Assumptions \ref{ass:5.1} and
  \ref{ass:5.2} be satisfied. Then
  Problem \ref{pro:5.1} has  at least
  an saddle point $(\bar u_(
  \cdot),\bar  u_2(\cdot))\in {\cal A}_1^S\times
  {\cal A}_2^S$  and
   \begin{eqnarray}
   \displaystyle\sup_{u_2(\cdot)\in{\cal
A}_2} \left(\inf_{u_1(\cdot)\in{\cal
A}_1}J(u_1(\cdot),u_2(\cdot))\right)=\inf_{u_1(\cdot)\in{\cal A}_1}(\sup_{u_2(\cdot)\in{\cal
A}_2}J(u_1(\cdot),u_2(\cdot))).
   \end{eqnarray}

\end{thm}
\begin{proof}

Since   ${\cal A}_1^S=
M_{{\mathscr F}_{Y}}^2(0,T; \mathbb R^{k_1})$ and
${\cal A}_2^S=
M_{{\mathscr F}_{Y}}^2(0,T; \mathbb R^{k_2})$
are s Hilbert spaces thus  reflexive Banach  spaces. Indeed, by the a priori
estimate  for the state process in
  Lemma 1.4 in Tang and Meng (2016), over
$ {\cal A}_1^S\times {\cal A}_2^S, $
 we can show that the cost functional
$J ( u_1 (\cdot), u_2(\cdot) )$ is continuous
 with respect to $(u_1 (\cdot), u_2(\cdot))$
 and hence lower-semi continuous with respect to
  $u_1(\cdot)$ and upper-semi continuous with respect to
  $u_2(\cdot),$  respectively. Since the weighting matrices in the cost
functional are not random,  from the
definition of  $J ( u_1 (\cdot), u_2(\cdot) )$ (see \eqref{eq:5.3}) and by a simple calculation, we can get that
  \begin{eqnarray}\label{eq:2.4}
\begin{split}
 J( u_1(\cdot), u_2(\cdot))=&
 \mathbb E\bigg[\int_0^T\big(\langle N_1(t)(u(t)-\mathbb E[u(t)]), u(t)-\mathbb E [u(t)]\rangle
 +\langle  (N_1(t)+
{N}_2(t))\mathbb E[u(t)], \mathbb E[u(t)]\rangle \big )dt\bigg]
\\&+\mathbb E\bigg[\langle
M, X(T)\rangle \bigg]+\mathbb E\bigg[\int_0^T\bigg(\langle Q(t),
X(t)\rangle \bigg)dt\bigg]
\\&+ {\mathbb E} \bigg [ \int_0^T
 \langle N_{21} (s) u_2 (s), u_2 (s) \rangle d s\bigg]+ {\mathbb E} \bigg [ \int_0^T \langle N_{22} (s)
\mathbb E[u_{2} (s)], \mathbb E[u_2 (s)] \rangle d s
\bigg ].
\end{split}
\end{eqnarray}
Thus  the cost functional $J(u_1(\cdot), u_2(\cdot))$ is convex  with respect to $u_1(\cdot)$ over
 ${\cal A}_1^S$ from the nonnegativity of the
 $N_{11}, N_{11}+N_{12}$.  Furthermore, it follows from the
uniformly strictly positivity  of $N_{11}, N_{11}+N_{12}$  and  the a priori
estimate  for the state process, that
\begin{eqnarray}\label{eq:2.5}
  \begin{split}
    J(u_1(\cdot), u_2(\cdot)) \geq &\delta \mathbb E\bigg[\int_0^T\langle u_1(t)-\mathbb E[u_1(t)], u_1(t)-\mathbb E[u_1(t)]\rangle dt\bigg]+
 \delta\mathbb E\bigg[\int_0^T \langle \mathbb E[u_1(t)], \mathbb E[u_1(t)]\rangle dt\bigg]-K||u_1(\cdot)||_{{\cal A}_1^S}
 \\&+ {\mathbb E} \bigg [ \int_0^T
 \langle N_{21} (s) u_2 (s), u_2 (s) \rangle d s\bigg]+ {\mathbb E} \bigg [ \int_0^T \langle N_{22} (s)
\mathbb E[u_{2} (s)], \mathbb E[u_2 (s)] \rangle d s
\bigg ].
 \\=& \delta||u_1(\cdot)||_{{\cal A}_1^S}^2
 -K||u_1(\cdot)||_{{\cal A}_1^S}
 \\&+ {\mathbb E} \bigg [ \int_0^T
 \langle N_{21} (s) u_2 (s), u_2 (s) \rangle d s\bigg]+ {\mathbb E} \bigg [ \int_0^T \langle N_{22} (s)
\mathbb E[u_{2} (s)], \mathbb E[u_2 (s)] \rangle d s
\bigg ],
  \end{split}
\end{eqnarray}
which implies
\begin{eqnarray*}
\lim_{ \| u_1 (\cdot) \|_{{\cal A}_1^S} {\rightarrow +\infty} } J ( u_1 (\cdot), u_2(\cdot) ) =+ \infty .
\end{eqnarray*}
On the other hand, we have
 \begin{eqnarray}\label{eq:2.4}
\begin{split}
 J( u_1(\cdot), u_2(\cdot))=&
 \mathbb E\bigg[\int_0^T\big(\langle N_{21}(t)(u_2(t)-\mathbb E[u_2(t)]),
  u_2(t)-\mathbb E [u_2(t)]\rangle
 +\langle  (N_{21}(t)+
{N}_{22}(t))\mathbb E[u_2(t)], \mathbb E[u_2(t)]\rangle \big )dt\bigg]
\\&+\mathbb E\bigg[\langle
M, X(T)\rangle \bigg]+\mathbb E\bigg[\int_0^T\langle Q(t),
X(t)\rangle dt\bigg]
\\&+ {\mathbb E} \bigg [ \int_0^T
 \langle N_{11} (s) u_1 (s), u_1 (s) \rangle d s\bigg]+ {\mathbb E} \bigg [ \int_0^T \langle N_{12} (s)
\mathbb E[u_{1} (s)], \mathbb E[u_1 (s)] \rangle d s
\bigg ].
\end{split}
\end{eqnarray}

Thus  the cost functional $J(u_1(\cdot), u_2(
\cdot))$ over
 ${\cal A}_2^S$ is concave with respect to
 $u_2(\cdot)$ from the negativity of the
 $N_{21}, N_{21}+N_{22},$.  Furthermore, it follows from the uniformly strictly negativity  of $N_{21}, N_{21}+N_{22},$, that
\begin{eqnarray}\label{eq:2.5}
  \begin{split}
    J(u_1(\cdot), u_2(\cdot)) \leq  &-\delta \mathbb E\bigg[\int_0^T\langle u_2(t)-\mathbb E[u_2(t)], u_2(t)-\mathbb E[u_2(t)]\rangle dt\bigg]+
 -\delta\mathbb E\bigg[\int_0^T \langle \mathbb E[u_2(t)], \mathbb E[u_2(t)]\rangle dt\bigg]+K||u_2(\cdot)||_{{\cal A}_1^S}
 \\&+ {\mathbb E} \bigg [ \int_0^T
 \langle N_{11} (s) u_1 (s), u_1 (s) \rangle d s\bigg]+ {\mathbb E} \bigg [ \int_0^T \langle N_{12} (s)
\mathbb E[u_{1} (s)], \mathbb E[u_1 (s)] \rangle d s
\bigg ]
 \\=& -\delta||u_2(\cdot)||_{{\cal A}_2^S}^2
 +K||u_2(\cdot)||_{{\cal A}_2^S}
 \\&+ {\mathbb E} \bigg [ \int_0^T
 \langle N_{11} (s) u_1(s), u_1 (s) \rangle d s\bigg]+ {\mathbb E} \bigg [ \int_0^T \langle N_{12} (s)
\mathbb E[u_{1} (s)], \mathbb E[u_1 (s)] \rangle d s
\bigg ],
  \end{split}
\end{eqnarray}
which implies
\begin{eqnarray*}
\lim_{ \| u_2 (\cdot) \|_{{\cal A}_2^S} {\rightarrow +\infty} } J ( u_1 (\cdot), u_2(\cdot) ) =- \infty .
\end{eqnarray*}

In summary, by Lemma \ref{lem:5.2}, Problem \ref{pro:5.1} has  at least
  an saddle point $(\bar u_(
  \cdot),\bar  u_2(\cdot))\in {\cal A}_1^S\times
  {\cal A}_2^S.$ The proof is complete.
\end{proof}

In the following,  applying the maximum
principle to our LQ differential game problem, we give the dual presentation of the optimal control in terms
of the corresponding adjoint process.

\begin{thm}\label{thm:5.3}
Let  Assumptions \ref{ass:5.1}and \ref{ass:5.2} be satisfied.
 Then, a necessary and
sufficient condition for an admissible pair $(u_1(\cdot),u_2(\cdot); x(\cdot))$ to be an optimal pair of  Problem (LQ) is the
control $(u_1(\cdot), u_2(\cdot))$ satisfies
\begin{eqnarray} \label{eq:5.13}
  \begin{split}
 &2N_{11}(t)u(t)+2N_{12}(t)\mathbb E[u_1(t)]+
 (B^\top_1(t)-h(t)G_{11}^\top(t))
 \mathbb E[p(t)|\mathscr F^{Y}_{t}]+(B^\top_{12}(t)
 -h(t)G^\top_{12}(t))\mathbb E [p(t)]
 \\&~~~~~~+ D^{\top}_{12}(t)\mathbb E[q(t)|\mathscr F^{Y}_{t}]
 +D^{\top}_{12}(t) \mathbb E [q(t)]=0, \quad a.e. a.s.,
     \end{split}
\end{eqnarray}
and
\begin{eqnarray} \label{eq:5.14}
  \begin{split}
 &2N_{21}(t)u(t)+2N_{22}(t)\mathbb E[u(t)]+
 (B^\top_1(t)-h(t)G_{21}^\top(t))
 \mathbb E[p(t)|\mathscr F^{Y}_{t}]+(B^\top_{22}(t)
 -h(t)G^\top_{22}(t))\mathbb E [p(t)]
 \\&~~~~~~+ D^{\top}_{21}(t)\mathbb E[q(t)|\mathscr F^{Y}_{t}]
 +D^{\top}_{22}(t) \mathbb E [q(t)]=0, \quad a.e. a.s.,
     \end{split}
\end{eqnarray}

where $(p(\cdot),q(\cdot), \tilde q(\cdot)) $ is the
solution to the adjoint equation \eqref{eq:5.7}
corresponding to $(u(\cdot), X(\cdot))$.
\end{thm}

\begin{proof}
 For the necessary part, let $(u_1(\cdot), u_2(\cdot), x(\cdot))$
 be an optimal pair associated with
 the adjoint process $(p(\cdot),q(\cdot), \tilde q(\cdot)).$   Since
 there is no constraints on the control processes, then  from the
 necessary optimality conditions \eqref{eq:4.10} and \eqref{eq:4.11}
 (see Theorem \ref{thm:4.5}), we get  that for
 $i=1,2,$
 \begin{eqnarray} \label{eq:5.10}
 \begin{split}
 &\mathbb E\bigg[H_{u_i}(t,x(t),\mathbb E[x(t)], u_1(t), \mathbb E[u_1(t)],u_2(t), \mathbb E[u_2(t)],, p(t), q(t), \tilde q(t))|\mathscr
F_t^{Y}\bigg]
\\&+\mathbb E\bigg[H_{v_i}(t,x(t),\mathbb E[x(t)], u_1(t), \mathbb E[u_1(t)],u_2(t), \mathbb E[u_2(t)], p(t), q(t), \tilde q(t))\bigg]=0,
\end{split}
 \end{eqnarray}
 which leads to
 \eqref{eq:5.14} and \eqref{eq:5.13} ( recalling  the definition  \eqref{eq:5.6}
  of Hamiltonian  $H$).

  For the sufficient part,  let $(u(\cdot), X(\cdot))$
 be an admissible pair associated with
 the adjoint process $(p(\cdot),q(\cdot), \tilde q(\cdot)) $ and assume  the conditions
 \eqref{eq:5.1} and \eqref{eq:5.14} hold.   From the
 definition of $H$ (see \eqref{eq:5.6}),
   the conditions
 \eqref{eq:5.13} and \eqref{eq:5.14} implies \eqref{eq:5.10} holds.
  Thus,  since   any
  admissible control is
  $ \mathscr F_{t}^Y$-adapted process.
   by\eqref{eq:5.10}
  , for any other
  admissible control $v_i(\cdot)\in {\cal A}_i^S,i=1,2,$
  from the property of conditional
  expectation, we have
  \begin{eqnarray} \label{eq:4.13}
  \begin{split}
  &\mathbb E\bigg[\bigg\langle  v_i (t) - u_i (t), H_{u_i}(t,x(t),\mathbb E[x(t)], u_1(t), \mathbb E[u_1(t)],u_2(t), \mathbb E[u_2(t)], p(t), q(t), \tilde q(t))
\\&~~~~~~+\mathbb E\bigg[H_{v_i}(t,x(t),\mathbb E[x(t)], u_1(t), \mathbb E[u_1(t)], u_2(t), \mathbb E[u_2(t)],p(t), q(t), \tilde q(t))\bigg]\bigg\rangle\bigg]
\\=&\mathbb E\bigg[\bigg\langle  v_i (t) - \bar u_i (t), \mathbb E\bigg[H_u(t,x(t),\mathbb E[x(t)], u_1(t), \mathbb E[u_1(t)],u_2(t), \mathbb E[u_2(t)], p(t), q(t), \tilde q(t))|\mathscr
F_t^{Y}\bigg]
\\&~~~~~~+\mathbb E\bigg[H_v(t,x(t),\mathbb E[x(t)], u_1(t), \mathbb E[u_1(t)],u_2(t), \mathbb E[u_2(t)], p(t), q(t), \tilde q(t))\bigg]\bigg\rangle\bigg]
\\&=0,
\end{split}
\end{eqnarray}
which implies that
the conditions \eqref{eq:4.8} and \eqref{eq:4.80} in Theorem \ref{thm:4.3}
holds. Moreover, under Assumptions \ref{ass:5.1}
and \ref{ass:5.2}, it is easy to check that all other conditions of (i) and (ii) in  Theorem \ref{thm:4.3} are satisfied.
Therefore, by (iii) in Theorem \ref{thm:4.3}, we conclude that $(u_1(\cdot), u_2(\cdot); x(\cdot))$ is an optimal loop-open control pair. The proof is complete.
\end{proof}
From the above result, we see that if Problem
\ref{pro:4.1}
admits an open-loop saddle point,  then the following FBSDE admits an adapted solution $(
u(\cdot), x(\cdot), p(\cdot),q(\cdot), \tilde q(\cdot)) $
\begin{equation}\label{eq:5.12}
\left\{\begin {array}{ll}
  dX(t)=&(A_1(t)X(t)+A_2(t)\mathbb E [X(t)]
  +B_{11}(t)u_1(t)+B_{12}(t)\mathbb E [u_1(t)]
  +B_{21}(t)u_2(t)+B_{22}(t)\mathbb E [u_2(t)])dt
  \\&+(C_1(t)X(t)
  + C_2(t)\mathbb E [X(t)]
  +D_{11}(t)u_1(t)+D_{12}(t)\mathbb E [u_1(t)]
  +D_{21}(t)u_2(t)+D_{22}(t)\mathbb E [u_2(t)])dW(t)
  \\&+(F_1(t)X(t)
  +F_2(t)\mathbb E [X(t)]
  +G_{11}(t)u_1(t)+G_{12}(t)\mathbb E [u_1(t)]
  +G_{21}(t)u_2(t)+G_{22}(t)\mathbb E [u_2(t)])dW^{(u_1,u_2)}(t),\\
  dY(t)  =& h(t)dt+dW^{(u_1, u_2)}(t), \\
  dp(t)=&-\bigg[(A^\top_1(t)-h(t)F_1^\top(t))p(t)+ (A_2^\top(t)-h(t)F_2^\top(t))\mathbb E [p(t)]+C^{\top}_1(t)q(t)
\\&+ C^{\top}_2(t)\mathbb E[q(t)]
+ F^{\top}_1(t)\tilde q(t)+F^{\top}_2(t)\mathbb E[\tilde q(t)]+Q(t)\bigg]dt+q(t)dW(t)+\tilde q(t)dY(t),
 \\x(0)=&x ,p(T)=M_, Y(0)=0\\
 2N_{i1}(t)u(t)&+2N_{i2}(t)\mathbb E[u_i(t)]+
 (B^\top_{i1}(t)-h(t)G_{i1}^\top(t))
 \mathbb E[p(t)|\mathscr F^{Y}_{t}]+(B^\top_{i2}(t)
 -h(t)G^\top_{i2}(t))\mathbb E [p(t)]
 \\&~~~~~~+ D^{\top}_{i2}(t)\mathbb E[q(t)|\mathscr F^{Y}_{t}]
 +D^{\top}_{i2}(t) \mathbb E [q(t)]=0, i=1,2,\quad a.e. a.s.
\end {array}
\right.
\end{equation}
Then by Theorem \ref{thm:5.3},
we can  directly  obtain the following   equivalence
 between the solvability of optimality system \eqref{eq:5.12} and the existence  of
 the optimal  open-loop
 control of Problem \ref{pro:5.1}.
\begin{cor}
  Let Assumptions \ref{ass:5.1} and
  \ref{ass:5.2} be satisfied.
 Then, a necessary and
sufficient condition for  that
the optimality system \eqref{eq:5.12}
has a    solution  $(u_1(\cdot), u_2(\cdot), x(\cdot),p(\cdot),
 q(\cdot),\tilde q(\cdot))\in
 M_{\mathscr{F}^Y}^2(0,
T;\mathbb R^n)\times M_{\mathscr{F}^Y}^2(0, T;\mathbb R^n)\times S_{\mathscr{F}}^2(0,
T;\mathbb R^n)\times S_{\mathscr{F}}^2(0, T;\mathbb R^n)\times M_{\mathscr{F}}^2(0,T;\mathbb R^{n})\times M_{\mathscr{F}}^2(0,T;\mathbb R^{n})$  is that
 $(u_1(\cdot), u_2(\cdot); x(\cdot))$ is  an
 optimal open-loop pair of  Problem \ref{pro:5.1}.
\end{cor}

\begin{rmk}
In summary, the optimality system \eqref{eq:5.12}
completely characterizes the optimal open-loop control of Problem \ref{pro:5.1}. Therefore, solving Problem \ref{pro:5.1} is equivalent to solving
the optimality system, moreover, the  optimal
open-loop control
can be given by \eqref{eq:5.13} and
\eqref{eq:5.14} which implies  that
the optimal loop-open control $(u_1(\cdot),u_2(\cdot))$
has the following explicit
dual presentation

\begin{eqnarray} \label{eq:4.17}
  \begin{split}
 \bar u_{1}(t)=&-\frac{1}{2}N^{-1}_{11}(t)\bigg\{
 (B^\top_{11}(t)-h(t)G_{11}^\top(t))
 \mathbb E[p(t)|\mathscr F^{Y}_{t}]+(B^\top_{12}(t)
 -h(t)G^\top_{12}(t))\mathbb E [p(t)]
 \\&~~~~~~+D^{\top}_{11}(t)\mathbb E[q(t)|\mathscr F^{Y}_{t}]
 +D^{\top}_{12}(t) \mathbb E [q(t)]
\\&+ N_{12}(t)(N_{11}(t)+ N_{12}(t))^{-1}\bigg[(B_{11} (t)+ B_{12} (t)-h(t)G_{11} (t)
 -h(t)G_{12} (t))^\top \mathbb E [p(t)]
 \\&\quad \quad\quad\quad\quad\quad\quad
 \quad+ (D_{11}(t)
 +D_{12}(t))^\top \mathbb E [q(t)]\bigg]\bigg\}
, \quad a.e. a.s.
     \end{split}
\end{eqnarray}

and

\begin{eqnarray} \label{eq:4.17}
  \begin{split}
 \bar u_{2}(t)=&-\frac{1}{2}N^{-1}_{21}(t)\bigg\{
 (B^\top_{21}(t)-h(t)G_{21}^\top(t))
 \mathbb E[p(t)|\mathscr F^{Y}_{t}]+(B^\top_{22}(t)
 -h(t)G^\top_{22}(t))\mathbb E [p(t)]
 \\&~~~~~~+D^{\top}_{21}(t)\mathbb E[q(t)|\mathscr F^{Y}_{t}]
 +D^{\top}_{22}(t) \mathbb E [q(t)]
\\&+ N_{22}(t)(N_{21}(t)+ N_{22}(t))^{-1}\bigg[(B_{21} (t)+ B_{22} (t)-h(t)G_{21} (t)
 -h(t)G_{22} (t))^\top \mathbb E [p(t)]
 \\&\quad \quad\quad\quad\quad\quad\quad
 \quad+ (D_{21}(t)
 +D_{22}(t))^\top \mathbb E [q(t)]\bigg]\bigg\}
, \quad a.e. a.s.
     \end{split}
\end{eqnarray}

\end{rmk}

\section{Conclusion}
In this paper, we have proved partial observed
stochastic maximum principle  for  the stochastic   differential
games  driven by  mean-field stochastic
differential equations.
     As an application, some  partial observed  linear quadratic stochastic differential game problem
      of mean-field type  is discussed and the
       the existence of  the open-loop saddle is obtained and the
       optimal control process is characterized
       explicitly by the adjoint process.
       Our main results could be seen as an extension of the stochastic optimal control problem
     studied   in Tang and Meng (2016), to the two-person zero-sum stochastic differential game problem.

\end{document}